\documentclass[runningheads]{llncs}
\usepackage{graphicx}
\usepackage[utf8]{inputenc}

\usepackage[ruled]{algorithm2e} 
\usepackage{tensor}
\usepackage{float}

\usepackage[font=small]{caption}
\usepackage{amsmath, amsthm,amsfonts,amssymb}
\usepackage{mathrsfs}
\usepackage[danish,english]{babel}
\usepackage{xcolor}
\usepackage{import}
\usepackage{subcaption}
\usepackage{hyperref}
\usepackage{cite}

\newcommand{\R}{\mathbb{R}}

\newcommand{\Q}{\mathbb{Q}}

\newcommand{\E}{\mathbb{E}}
\newcommand{\Pb}{\mathbb{P}}
\newcommand{\Qb}{\mathbb{Q}}

\newcommand{\cF}{\mathcal{F}}

\newcommand{\cG}{\mathcal{G}}

\newcommand{\cFM}{\mathcal{FM}}

\newcommand{\partiel}[1]{\frac{\partial}{\partial {#1}}}

\DeclareMathOperator{\Log}{Log}
\DeclareMathOperator{\Exp}{Exp}
\DeclareMathOperator{\Vol}{Vol}
\DeclareMathOperator{\Cut}{Cut}

\usepackage{mathtools}
\DeclarePairedDelimiterX{\norm}[1]{\lVert}{\rVert}{#1}

\newcommand{\reflemma}[1]{Lemma~\ref{#1}}

\newcommand{\refsec}[1]{Section~\ref{#1}}

\DeclareMathOperator{\Ad}{\mathrm{Ad}}

\newcommand{\SO}{\mathrm{SO}}

\newcommand{\Symp}{\mathrm{Sym^+}}

\begin{document}
\title{Bridge Simulation and Metric Estimation on Lie Groups}
%
%
\author{Mathias H\o jgaard Jensen\inst{1} \and
Sarang Joshi\inst{2} \and Stefan Sommer\inst{1}}
\authorrunning{Jensen et al.}
%
\institute{Department of Computer Science, University of Copenhagen\\
Universitetsparken 1, 2100, Copenhagen \O, Denmark\\
\email{\{matje,sommer\}@di.ku.dk} \and
Department of Biomedical Engineering, University of Utah\\
72 S Central Campus Drive, Salt Lake City, UT 84112, USA\\
\email{sjoshi@sci.utah.edu}}
\maketitle              

\vspace{-.5cm}

\begin{abstract}
We present a simulation scheme for simulating Brownian bridges on complete and connected Lie groups. We show how this simulation scheme leads to absolute continuity of the Brownian bridge measure with respect to the guided process measure. This result generalizes the Euclidean result of Delyon and Hu to Lie groups. We present numerical results of the guided process in the Lie group $\SO(3)$. In particular, we apply importance sampling to estimate the metric on $\SO(3)$ using an iterative maximum likelihood method.

\end{abstract}

\vspace{-1cm}

\section{Introduction}

\vspace{-3mm}

Bridge simulation techniques are known to play a fundamental role in statistical inference for diffusion processes. Diffusion bridges in manifolds have mainly been used to provide gradient and hessian estimates. To the best of our knowledge, this paper is the first to describe a simulation technique for diffusion bridges in the context of Lie groups.

The paper is organized as follows. In \refsec{sec: notation and background}, we describe some background theory of Lie groups, Brownian motions, and Brownian bridges in Riemannian manifolds. \refsec{sec: simulation of bridges on Lie groups} presents the theory and results. \refsec{sec: importance sampling and metric estimation} shows in practice the simulation scheme in the Lie group $\SO(3)$. Using importance sampling, we obtain an estimate of the underlying unknown metric.

\vspace{-4mm}

\section{Notation and Background}\label{sec: notation and background}

\vspace{-3mm}
\subsubsection{Lie Groups}

Throughout, we let $G$ denote a connected Lie Group of dimension $d$, i.e., a smooth manifold with a group structure such that the group operations $G \times G \ni (x,y) \overset{\mu}{\mapsto} xy \in G$ and $G \ni x \overset{\iota}{\mapsto} x^{-1}\in G$ are smooth maps. If $x \in G$, the left-multiplication map, $L_x y$, defined by $y \mapsto \mu(x,y)$, is a diffeomorphism from $G$ to itself. Similarly, the right-multiplication map $R_x y$ defines a diffeomorphism from $G$ to itself by $y \mapsto \mu(y,x)$. 
We assume throughout that $G$ acts on itself by left-multiplication. Let $dL_x \colon TG \rightarrow TG$ denote the pushforward map given by $(dL_x)_y \colon T_yG \rightarrow T_{xy}G$. A vector field $V$ on $G$ is said to be left-invariant if $(dL_x)_yV(y) = V(xy)$. The space of left-invariant vector fields is linearly isomorphic to $T_eG$, the tangent space at the identity element $e \in G$. By equipping the tangent space $T_eG$ with the Lie bracket we can identify the Lie algebra $\cG$ with $T_eG$. The group structure of $G$ makes it possible to define an action of $G$ on its Lie algebra $\cG$. The conjugation map $C_x := L_x \circ R_x^{-1} \colon y \mapsto xyx^{-1}$, for $x \in G$, fixes the identity $e$. Its pushforward map at $e$, $(dC_x)_e$, is then a linear automorphism of $\cG$. Define $\Ad(x) := (dC_x)_e$, then $\Ad \colon x \mapsto \Ad(x)$ is the adjoint representation of $G$ in $\cG$. The map $G \times \cG \ni (x,v) \mapsto \Ad(x)v \in \cG$ is the adjoint action of $G$ on $\cG$. We denote by $\langle \cdot , \cdot \rangle$ a Riemannian metric on $G$. The metric is said to be left-invariant if $\langle u, v \rangle_y = \left\langle (dL_x)_y u, (dL_x)_y v \right\rangle_{L_x(y)}$, for every $u,v \in T_yG$, i.e., the left-multiplication maps are isometries, for every $x \in G$. In particular, we say that the metric is $\Ad(G)$-invariant if  $\langle u, v \rangle_e = \left\langle \Ad(x) u, \Ad(x) v \right\rangle_e$, for every $u,v \in \cG$. Note that an $\Ad(G)$-invariant inner on $\cG$ induces a bi-invariant (left- and right-invariant) metric on $G$.

\vspace{-5mm}

\subsubsection{Brownian Motion}

Endowing a smooth manifold $M$ with a Riemannian metric, $g$, allows us to define the Laplace-Beltrami operator, $\Delta_M f =$ div grad $f$. This operator is the generalization of the Euclidean Laplacian operator to manifolds. In terms of local coordinates $(x_1,\dots,x_d)$ the expression for the Laplace-Beltrami operator becomes
$\Delta_M f = \det(g)^{-1/2}\left(\frac{\partial}{\partial x_j} g^{ji} \det(g)^{1/2} \frac{\partial}{\partial x_i}\right)f$,
where $\det(g)$ denotes the determinant of the Riemannian metric $g$ and $g^{ij}$ are the coefficients of the inverse of $g$. An application of the product rule implies that $\Delta_M$ can be rewritten as
$
\Delta_Mf = a^{ij}\frac{\partial}{\partial x_i} \frac{\partial}{\partial x_j} f + b^j \frac{\partial}{\partial x_j} f,
$
where $a^{ij}=g^{ij}$, $b^k = - g^{ij}\Gamma^k_{ij}$, and $\Gamma$ denote the Christoffel symbols related to the Riemannian metric. This diffusion operator defines a Brownian motion on the $M$, valid up to its first exit time of the local coordinate chart. 

In the case of the Lie group $G$, the identification of the space of left-invariant vector fields with the Lie algebra $\cG$ allows for a global description of $\Delta_G$. Indeed, let $\{v_1,\dots v_d\}$ be an orthonormal basis of $T_eG$. Then $V_i(x) = (dL_x)_e v_i$ defines left-invariant vector fields on $G$ and the Laplace-Beltrami operator can be written as (cf. \cite[Proposition 2.5]{liao2004levy})
    $
    \Delta_G f(x) = \sum_{i=1}^d V_i^2f(x) - V_0 f(x),
   $
where $V_0 = \sum_{i,j=1}^d C_{ij}^j V_j$ and $C^k_{ij}$ denote the structure coefficients given by $[V_i,V_j] = C^k_{ij}V_k$. The corresponding stochastic differential equation (SDE) for the Brownian motion on $G$, in terms of left-invariant vector fields, then becomes
    \begin{equation}\label{eq: BM SDE on Lie group}
        dX_t = - \frac{1}{2}V_0(X_t) dt + V_i(X_t) \circ dB^i_t, \qquad X_0 = e,
    \end{equation}
where $\circ$ denotes integration in the Stratonovich sense. By \cite[Proposition 2.6]{liao2004levy}, if the inner product is $\Ad(G)$ invariant, then $V_0 = 0$. The solution of \eqref{eq: BM SDE on Lie group} is conservative or non-explosive and is called the left-Brownian motion on $G$ (see \cite{shigekawa1984transformations} and references therein).

\vspace{-5mm}
\subsubsection{Riemannian Brownian Bridges}

In this section, we briefly review some classical facts on Brownian bridges on Riemannian manifolds. As Lie groups themselves are manifolds, the theory carries over \textit{mutatis mutandis}. However, Lie groups' group structure allows the notion of left-invariant vector fields. The identification of the Lie algebra with the vector space of left-invariant vector fields makes Lie groups parallelizable. Thus, the frame bundle construction for developing stochastic processes on manifolds becomes superfluous since left-invariant vector fields ensure stochastic parallel displacement.

Let $\Pb^t_x$ be the measure of a Riemannian Brownian motion, $X_t$, at some time $t$ started at point $x$. Suppose $p$ denotes the transition density of the Riemannian Brownian motion. In that case, $d\Pb^t_x = p(t,x,y) d\Vol(y)$ describes the measure of the Riemannian Brownian motion, where $d\Vol(y)$ is the Riemannian volume measure. Conditioning the Riemannian Brownian motion to hit some point $v$ at time $T>0$ results in a Riemannian Brownian bridge. Here, $\Pb_{x,v}^T$ denotes the corresponding probability measure. The two measures are equivalent over the time interval $[0, T)$, however mutually singular at time $t=T$. The initial enlargement of the filtration remedies the singularity. The corresponding Radon-Nikodym derivative is given by 
	\begin{equation*}
		\frac{d\Pb_{x,v}^T}{d\Pb_x^T}\big|_{\cF_s} = \frac{p(T-s,X_s, v)}{p(T,x,v)} \qquad \text{for } 0 \leq s <T,
	\end{equation*}
which is a martingale for $s < T$. The Radon-Nikodym derivative defines the density for the change of measure and provides the basis for the description of Brownian bridges. In particular, it provides the conditional expectation defined by
	\begin{equation*}
		\E[F(X_t)|X_T = v] = \frac{\E[p(T-t,X_t,v)F(X_t)]}{p(T,x,v)},
	\end{equation*}
for any bounded and $\cF_s$-measurable random variable $F(X_s)$. As described in \cite{hsu2002stochastic}, the Brownian bridge yields an SDE in the frame bundle, $\cFM$, given by
	\begin{equation}\label{eq: BB sde on frame bundle}
	dU_t = H_i(U_t) \circ \left(dB^i_t + H_i \log \tilde{p}(T-t,U_t,v) dt \right), \qquad U_0 = u_0, 
	\end{equation}
in terms of the horizontal vector fields $(H_i)$, which is the lifted $M$-valued Brownian bridge, $X_t := \pi(U_t)$, where $\pi \colon \cFM \rightarrow M$.

\vspace{-5mm}

\section{Simulation of Bridges on Lie Groups}\label{sec: simulation of bridges on Lie groups}

\vspace{-4mm}

In this section, we consider the task of simulating \eqref{eq: BM SDE on Lie group} conditioned to hit $v \in G$, at time $T > 0$. The potentially intractable transition density for the solution of \eqref{eq: BM SDE on Lie group} inhibits simulation directly from \eqref{eq: BB sde on frame bundle}. Instead, we propose to add a guiding term mimicking that of Delyon and Hu \cite{delyon_simulation_2006}, i.e., the guiding term becomes the gradient of the distance to $v$ divided by the time to arrival. The SDE for the guided diffusion becomes
    \begin{equation}\label{eq: guided diffusion sde}
        dY_t = - \frac{1}{2}V_0(Y_t) dt + V_i(Y_t) \circ \left(dB^i_t - \frac{\left(\nabla_y d(Y_t,v)^2\right)^i}{2(T-t)}dt\right), \qquad Y_0 = e,
    \end{equation}
where $d$ denotes the Riemannian distance function. Note that we can always, for convenience, take the initial value to be the identity $e$. 

\vspace{-5mm}
\subsubsection{Radial Process}
We denote by $r_v(\cdot) := d(\cdot ,v)$ the radial process. Due to the radial process's singularities on $\Cut(v) \cup \{v\}$, the usual It\^o's formula only applies on subsets away from the cut-locus. The extension beyond the cut-locus of a Brownian motion's radial process was due to Kendall \cite{kendall1987radial}. Barden and Le \cite{barden1997some, le1995ito} generalized the result to $M$-semimartingales. The radial process of the Brownian motion \eqref{eq: BM SDE on Lie group} is given by
    \begin{equation}\label{eq: radial process for BM}
        r(X_t) = r(X_0)^2 + \int_0^t \left\langle \nabla r(X_s), V(X_s)dB_s\right\rangle + \frac{1}{2}\int_0^t \Delta_G r(X_s) ds - L_s(X),
    \end{equation}
where $L$ is the geometric local time of the cut-locus $\Cut(v)$, which is non-decreasing continuous random functional increasing only when $X$ is in $\Cut(v)$ (see \cite{barden1997some,kendall1987radial,le1995ito}). Let $W_t := \int_0^t \left\langle \partiel r,  V_i(X_s) \right\rangle dB^i_s$, which is the local-martingale part in the above equation. The quadratic variation of $W_t$ satisfies
    $
        d[W,W]_t 
        = 
     dt,
   $
by the orthonormality of $\{V_1,\dots, V_d\}$, thus $W_t$ is a Brownian motion by Levy's characterization theorem. From the stochastic integration by parts formula and \eqref{eq: radial process for BM} the squared radial process of $X$ satisfies 
    \begin{equation}\label{eq: squared radial process of BM}
         r(X_t)^2 = r(X_0)^2 + 2\int_0^t r(X_s) dW_s + \int_0^t r(X_s) \Delta_G r(X_s) ds - 2 \int_0^t r(X_s)dL_s,
    \end{equation}
where $dL_s$ is the random measure associated to $L_s(X)$.

Similarly, we obtain an expression for the squared radial process of $Y$. Using the shorthand notation $r_t:=r_v(Y_t)$ the radial process then becomes
    \begin{equation}\label{eq: squared radial process of guided diffusion}
         r_t^2 = r_0^2 + 2\int_0^t r_s dW_s  + \int_0^t \frac{1}{2}\Delta_G r_s^2 ds - \int_0^t \frac{r_s^2}{T-s} ds - 2\int_0^t r_s dL_s.
    \end{equation}
Imposing a growth condition on the radial process yields an $L^2$-bound on the radial process of the guided diffusion, \cite{thompson2018brownian}. So assume there exist constants $\nu \geq 1$ and $\lambda \in \R$ such that $\tfrac{1}{2} \Delta_G r_v^2 \leq \nu + \lambda r_v^2$ on $D\backslash \Cut(v)$, for every regular domain $D \subseteq G$. Then \eqref{eq: squared radial process of guided diffusion} satisfies
\begin{equation}\label{eq: mean squared radial bound} 
	\E [1_{t < \tau_D}r_v(Y_t)^2] \leq \left( r^2_v(e) + \nu t\left(\frac{t}{T-t}\right) \right)\left(\frac{T-t}{t}\right)^2 e^{\lambda t},
\end{equation}
where $\tau_D$ is the first exit time of $Y$ from the domain $D$.

\vspace{-5mm}
\subsubsection{Girsanov Change of Measure}

Let $B$ be the Brownian motion in $\R^d$ defined on the filtered probability space $(\Omega, \cF, (\cF_s),\Pb)$ and $X$ the solution of \eqref{eq: BM SDE on Lie group}. The process $\tfrac{\nabla r_v(X_t)^2}{2(T-t)}$ is an adapted process. As $X$ is non-explosive, we see that
    \begin{equation}\label{eq: novikov condition type bound on radial process}
    \int_0^t \left\lVert \frac{\nabla r(X_s)^2}{2(T-s)} \right\rVert^2 ds = \int_0^t \frac{r(X_s)^2 }{(T-s)^2}ds \leq C,
    \end{equation}
for every $0\leq t<T$, almost surely, and for some fixed constant $C > 0$. Define a new measure $\Qb$ by
    \begin{equation}\label{eq: radon-nikodym derivative}
       Z_t :=  \frac{d\Qb}{d\Pb}\bigg|_{\cF_t}(X)  = \exp\left\{-\int_0^t \left\langle  \frac{\nabla r(X_s)^2}{2(T-s)}, V(X_t)dB_s\right\rangle - \frac{1}{2} \int_0^t \frac{r(X_s)^2 }{(T-s)^2}ds \right\}.
    \end{equation}
From \eqref{eq: novikov condition type bound on radial process}, the process $Z_t$ is a martingale, for $t \in [0, T)$, and $\Qb_t$ defines a probability measure on each $\cF_t$ absolutely continuous with respect to $\Pb$. By Girsanov's theorem (see e.g. \cite[Theorem 8.1.2]{hsu2002stochastic}) we get a new process $b_s$ which is a Brownian motion under the probability measure $\Qb$. Moreover, under the probability $\Qb$, equation \eqref{eq: BM SDE on Lie group} becomes 
    \begin{equation}\label{eq: guided diffusion sde with radial unit vector}
        dY_t = - \frac{1}{2}V_0(Y_t)dt + V_i(Y_t) \circ \left(db^i_t - \frac{r(Y_t)}{T-t} \left(\partiel r\right)^i dt\right),
    \end{equation}
where $\left(\partiel r\right)^i$ is the $i$'th component of the unit radial vector field in the direction of $v$. The squared radial vector field is smooth away from $\Cut(v)$ and thus we set it to zero on $\Cut(v)$. Away from $\Cut(v)$, the squared radial vector field is $2\Log_v$, which is the inverse exponential at $v$. The added drift term acts as a guiding term, which pulls the process towards $v$ at time $T>0$. 

From \eqref{eq: radon-nikodym derivative}, we see that $\E[f(Y_t)] = \E[f(X_t)Z_t]$. Using \eqref{eq: squared radial process of BM} and the identity $\Delta_G r_v = \frac{d-1}{r_v} +  \partiel {r_v} \log \Theta_v$ (see \cite{thompson2015submanifold}), we equivalently write $\E[f(Y_t)\varphi_t] = \E[f(X_t)\psi_t]$, with
    \begin{equation}\label{eq: psi and phi function}
        \psi_t := \exp\left\{\frac{-r(X_t)^2}{2(T-t)}\right\} 
        \qquad \varphi_t := \exp\left\{ \int_0^t \frac{r_v(Y_s)^2}{T-s}\left(dA_s + dL_s\right)\right\},
    \end{equation}
where $dA_s = \partiel {r_v} \log \Theta_v$ is a random measure supported on $G \backslash \Cut(v)$ and $\Theta_v$ is the Jacobian determinant of $\Exp_v$.

\vspace{-5mm}

\subsubsection{Delyon and Hu in Lie Groups}

This section generalizes the result of Delyon and Hu \cite[Theorem 5]{delyon_simulation_2006} to the Lie group setting. The result can be modified to incorporate a generalization of \cite [Theorem 6]{delyon_simulation_2006}. 

    \begin{theorem}\label{result: delyon and hu result}
    Let $X$ be the solution of \eqref{eq: BM SDE on Lie group}. The SDE \eqref{eq: guided diffusion sde} yields a strong solution on $[0,T)$ and satisfies $\lim_{t\uparrow T} Y_t = v$ almost surely. Moreover, the conditional expectation of $X$ given $X_T=v$ is
    \begin{equation}
        \E[f(X)|X_T = v] = C \E\left[f(Y) \varphi_T\right],
    \end{equation}
    for every $\cF_t$-measurable non-negative function $f$ on $G$, $t<T$, where $\varphi_t$ is given in \eqref{eq: psi and phi function}.
    \end{theorem}
    
    \begin{proof}
        The result is a consequence of the change of measure together with \reflemma{result: almost sure convergence of guided bridge}, \reflemma{result: lemma of conditional expectation}, and \reflemma{result: lemma L1 convergence}.
    \end{proof}

    \begin{lemma}\label{result: almost sure convergence of guided bridge}
    The solution of SDE \eqref{eq: guided diffusion sde} satisfies $\lim_{t \rightarrow T} Y_t =v$ almost surely.
    \end{lemma}
    
    \begin{proof}
Let $\{D_n\}_{n=1}^{\infty}$ be an exhaustion of $G$, that is, the sequence consists of open, relatively compact subsets of $M$ such that $\bar{D}_n \subseteq D_{n+1}$ and $G = \bigcup_{n=1}^{\infty} D_n$. Furthermore, let $\tau_{D_n}$ denote the first exit time of $Y$ from $D_n$, then from \eqref{eq: mean squared radial bound} we have that the sequence 
$
	\bigl(\E [1_{\{t < \tau_{D_n}\}} r_v^2(Y_t)]\bigr)_{n=1}				^{\infty}
$
is non-decreasing and bounded, hence from the monotone convergence theorem, it has a limit which is bounded by the right-hand side of \eqref{eq: mean squared radial bound}. Applying Jensen's inequality to the left-hand side of \eqref{eq: mean squared radial bound} 
\begin{equation*}
	\E [ r_v(Y_t)] \leq \left( r^2_v(e) + \nu t\left(\frac{t}{T-t}\right) \right)^{\frac{1}{2}}\left(\frac{T-t}{t}\right) e^{\frac{\lambda t}{2}}.
\end{equation*}
Since obviously $\E [ r_v(Y_T)] = r_v(Y_T)\Qb(r_v(Y_T) \neq 0)$, by Fatou's lemma
    $
        \E[r_v(Y_T)] \leq \liminf_{t\rightarrow T} \E[r(Y_t)] = 0,
    $
we conclude that $r(Y_t) \rightarrow 0$, $\Q$-almost surely.    
    \end{proof}

    \begin{lemma}\label{result: lemma of conditional expectation}
    Let $0 < t_1 <t_2< \dots < t_N <T$ and $h$ be a continuous bounded function function on $G^N$. With $\psi_t$ as in \eqref{eq: psi and phi function}, then
    \begin{equation}
        \lim_{t \rightarrow T} \frac{\E\left[h\left(X_{t_1},X_{t_2},\dots, X_{t_N}\right)\psi_t\right]}{\E[\psi_t]} = \E\left[h\left(X_{t_1},X_{t_2},\dots, X_{t_N}\right)|X_T =v\right].
    \end{equation}
    \end{lemma}

    \begin{proof}
    The proof is similar to that of \cite[Lemma 7]{delyon_simulation_2006}. Let $(U,\phi)$ be a normal chart centered at $v \in G$. First, since the cut locus of any complete connected manifold has (volume) measure zero, we can integrate indifferently in any normal chart. For any $t \in (t_N,T)$ we have
	\begin{equation}
		\E[h(x_{t_1},...,x_{t_N})\psi_t] = \int_G \Phi_h(t,z)e^{-\frac{r_v(z)^2}{2(T-t)}}d\Vol(z)
	\end{equation}
where $d\Vol(z)= \sqrt{\det(A(z))}dz$ denotes the volume measure on $G$, $dz$ the Lebesgue measure, and $A$ the metric tensor. Moreover,
	\begin{equation*}
		\Phi_h(t,z) =
		 \int_{G^N}	h(z_1,...,z_N)p(t_1,u,z_1)\dots p(t-t_N,z_N,z)d			\Vol(z_1) \dots d\Vol(z_N),
	\end{equation*}
and of course $\Phi_1(t,z) = p(t,e,z)$. 
Using the normal chart and applying the change of variable $x=(T-t)^{1/2}y$ we get 
\begin{equation*}
	(T-t)^{-\frac{d}{2}} \E[h(x_{t_1},...,x_{t_N})\psi_t]
	\overset{t \rightarrow T}{\rightarrow}  \Phi_h(T,v) \det(A(v))^{\tfrac{d}{2}}\int_{\phi(G)}e^{-\frac{r_v(\phi^{-1}(y))^2}{2}}dy.
\end{equation*}
The conclusion follows from Bayes' formula.
    \end{proof}

\vspace{-6mm}

	\begin{lemma}\label{result: lemma L1 convergence}
	With $\varphi_t$ as defined above then $\varphi_t \overset{L_1}{\rightarrow} \varphi_T$.
	\end{lemma}
	
	\vspace{-3mm}
	
	\begin{proof}
	Note that for each $t \in [0,T)$ we have $\E^{\Qb}[\varphi_t]<\infty$ as well as $\varphi_t \overset{}{\rightarrow} \varphi_T$ almost surely by \reflemma{result: almost sure convergence of guided bridge}. The result then follows from the uniform integrability of $\left\{\varphi_t \colon t \in [0,T)\right\}$, which can be found in Appendix C.2 in \cite{thompson2015submanifold}.
	\end{proof}


\section{Importance Sampling and Metric Estimation on $\SO(3)$}\label{sec: importance sampling and metric estimation}


This section takes $G$ to be the special orthogonal group of rotation matrices, $\SO(3)$, a compact connected matrix Lie group. In the context of matrix Lie groups, computing left-invariant vector fields is straightforward. 
\vspace{-8mm}

\begin{figure}[H]
    \centering
    \begin{subfigure}[t]{.3\textwidth}
        \centering
        \includegraphics[width=\linewidth]{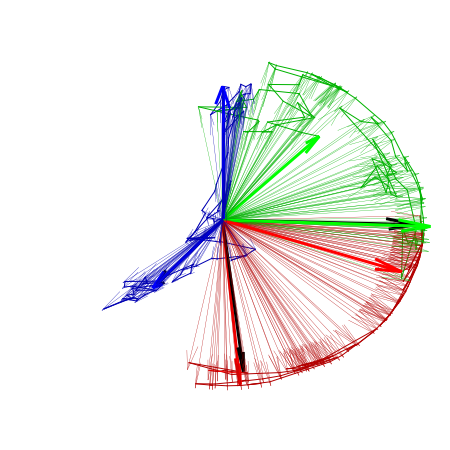}
        \caption{}
    \end{subfigure}
    \hfill
    \begin{subfigure}[t]{.3\textwidth}
        \centering
        \includegraphics[width=\linewidth]{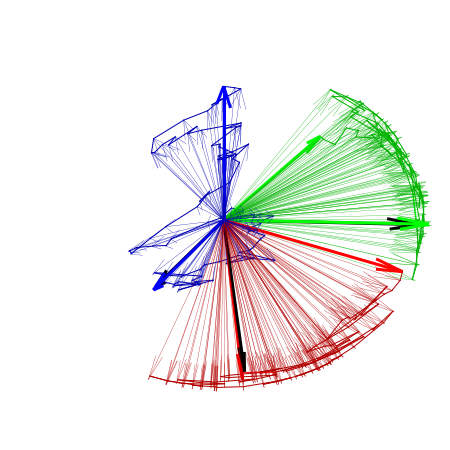}
        \caption{}
    \end{subfigure}
    \hfill
    \begin{subfigure}[t]{.3\textwidth}
        \centering
        \includegraphics[width=\linewidth]{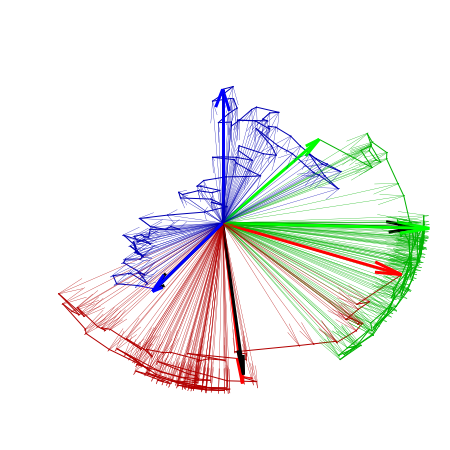}
        \caption{}
    \end{subfigure}
    \caption{Three sample paths $(a)-(c)$ of the guided diffusion process on $\SO(3)$ visualized by its action on the basis vectors $\{e_1,e_2,e_3\}$ (red, green, blue) of $\R^3$. The sample paths are conditioned to hit the rotation represented by the black vectors.}
    \label{fig: bridge samples on SO(3)}
\end{figure}

\vspace{-12mm}

\subsubsection{Numerical Simulations} 
The Euler-Heun scheme leads to approximation of the Stratonovich integral. With a time
discretization $t_1,\ldots,t_k$, $t_k-t_{k-1}=\Delta t$ and corresponding noise
$\Delta B_{t_i}\sim N(0,\Delta t)$, the numerical approximation of the Brownian motion \eqref{eq: BM SDE on Lie group} takes the form
\begin{equation}
  x_{t_{k+1}} = 
  x_{t_k}
  -
  \frac{1}{2}\sum_{j,i}C\indices{^j_{ij}}V_i(x_{t_k})\Delta t
  +
  \frac{v_{t_{k+1}}+V_i(v_{t_{k+1}}+x_{t_{k+1}})\Delta B_{t_k}^i}{2}
\label{eq: numerical brownian motion}
\end{equation}
where $v_{t_{k+1}} = V_i(x_{t_k})\Delta B_{t_k}^i$ is only used as an intermediate value in integration. Adding the logarithmic term in \eqref{eq: guided diffusion sde with radial unit vector} to \eqref{eq: numerical brownian motion} we obtain a numerical approximation of a guided diffusion \eqref{eq: guided diffusion sde}. Fig. \ref{fig: bridge samples on SO(3)} shows three different sample paths from the guided diffusion conditioned to hit the rotation represented by the black vectors.


\vspace{-3mm}

\subsubsection{Metric Estimation on $\SO(3)$}
 In the $d$-dimensional Euclidean case, importance sampling yields the estimate \cite{papaspiliopoulos2012importance}
    \begin{equation*}
        p(T,u,v) = \left(\frac{\det\left(A(T,v)\right)}{2 \pi T}\right)^{\tfrac{d}{2}} e^{-\frac{\lVert u-v\rVert^2_A}{2T}} \E[\varphi_T],
    \end{equation*}
where $\lVert x \rVert_A = x^TA(0,u)x$. Thus, from the output of the importance sampling we get an estimate of the transition density.
Similar to the Euclidean case, we obtain an expression for the heat kernel $p(T,e,v)$ as $p(T,e,v) = q(T,e)  \E\left[\varphi_T\right]$, where
	\begin{equation}\label{eq: q function in density}
	q(T,e) = \left(\frac{\det A(v)}{2\pi T}\right)^{\tfrac{3}{2}} \exp\left(-\frac{d(e,v)^2}{2T}\right) 
	= 
	\left(\frac{\det A(v)}{2\pi T}\right)^{\tfrac{3}{2}} \exp\left(-\frac{\lVert \Log_v(e)\rVert^2_A }{2T}\right),
	\end{equation}
where the equality holds almost everywhere and $A \in \Symp(\cG)$ denotes the metric $A(e)$. The $\Log_v$ map in \eqref{eq: q function in density} is the Riemannian inverse exponential map. 

Figure \ref{fig: Metric estimation on SO(3)} illustrate how importance sampling on $\SO(3)$ leads to metric estimation of the underlying true metric, from which the Brownian motion was generated. We sampled $128$ points as endpoints of a Brownian motion from the metric $\text{diag}(0.2, 0.2, 0.8)$. We used $20$ time steps and sampled $4$ bridges per observation. An iterative maximum likelihood method using gradient descent with a learning rate of $0.2$, and initial guess of the metric being $\text{diag}(1, 1, 1)$ yielded a convergence to the true metric. Note that in iteration the logarithmic map changes.

\vspace{-9mm}
\begin{figure}[H]
    \centering
    \begin{subfigure}[t]{.49\textwidth}
        \centering
        \includegraphics[width=.99\linewidth, height=145 pt]{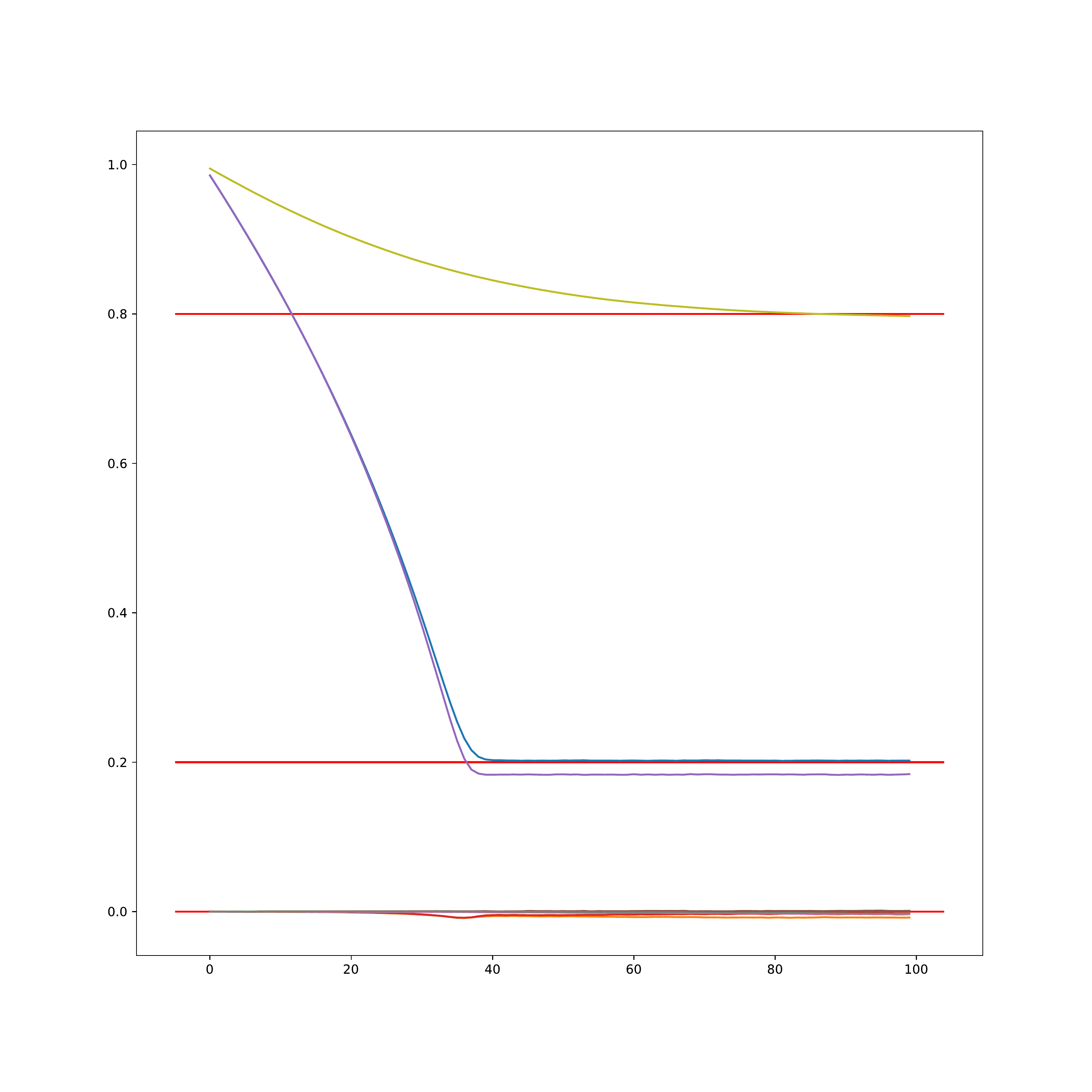}
        \caption{Estimation of the unknown underlying metric using bridge sampling. Here the true metric is the diagonal matrix $\text{diag}(0.2, 0.2, 0.8)$.}
    \end{subfigure}
    \hfill
    \begin{subfigure}[t]{.49\textwidth}
        \centering
        \includegraphics[width=.99\linewidth, height=145 pt]{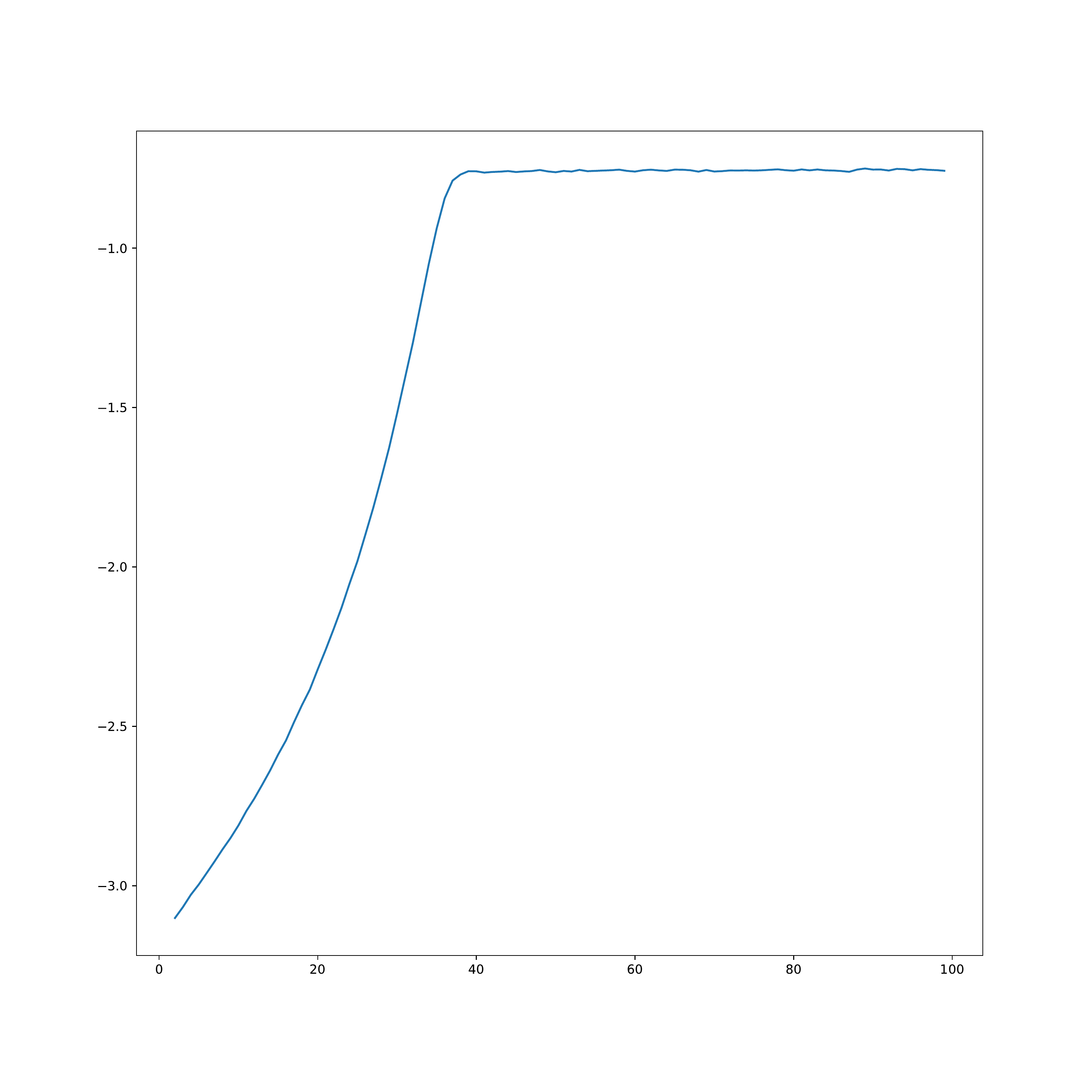}
        \caption{The iterative log-likelihood. }\label{fig: iterative likelihood}
    \end{subfigure}
    \caption{The importance sampling technique applies to metric estimation on the Lie group $\SO(3)$. Sampling a Brownian motion from an unknown underlying metric we obtain a convergence to the true underlying metric using an iterative maximum-likelihood method. Here we sampled $4$ bridge processes per observation, starting from the metric $\text{diag}(1, 1, 1)$, providing a relatively smooth iterative likelihood in \ref{fig: iterative likelihood}.}
    \label{fig: Metric estimation on SO(3)}
\end{figure}

%

\vspace{-14mm}

\bibliographystyle{splncs04}
\bibliography{bibfile}

\end{document}